\documentclass[10.5pt, reqno]{amsart}
\usepackage{amssymb, amsmath, amsthm}
\numberwithin{equation}{section}

\def\pasdegrille{\let\grille = \pasgrille}
\def\ecriture#1#2{\setbox1=\hbox{#1}
\dimen1= \wd1
\dimen2=\ht1
\dimen3=\dp1
\grille #2 \box1 }
\def\aat#1#2#3{
\divide \dimen1 by 48
\dimen3=\dimen1
\multiply \dimen1 by #1
\advance \dimen1 by -\dimen3
\divide \dimen1 by 101
\multiply \dimen1 by 100
\divide \dimen2 by \count11
\multiply \dimen2 by #2
\setbox0=\hbox{#3}\ht0=0pt\dp0=0pt
  \rlap{\kern\dimen1 \vbox to0pt{\kern-\dimen2\box0\vss}}\dimen1= \wd1
\dimen2=\ht1}
\def\pasgrille{
\count12= \dimen1
\divide \count12 by 50
\divide \dimen2 by \count12
\count11 =\dimen2
\
\divide \dimen1 by 48
\setlength{\unitlength}{\dimen1}
\smash{\rlap{\ }}
\dimen1= \wd1
\dimen2=\ht1
}
\def\grille{
\count12= \dimen1
\divide \count12 by 50
\divide \dimen2 by \count12
\count11 =\dimen2
\
\divide \dimen1 by 48
\setlength{\unitlength}{\dimen1}
\smash{\rlap{\graphpaper[1](0,0)(50, \count11)}}
\dimen1= \wd1
\dimen2=\ht1
}

\pasdegrille
\usepackage{amssymb}
\usepackage{amsmath, amsthm, amsopn, amsfonts}
\usepackage[dvips]{graphicx}
\usepackage{graphpap}
\usepackage[english, francais]{babel}
\newtheorem{theoreme}{Theorem}

\newtheorem{proposition}{Proposition}
\newtheorem{lemme}[proposition]{Lemma}
\newtheorem{corollaire}[proposition]{Corollary}
\newtheorem{remarque}[proposition]{Remark}

\numberwithin{equation}{section}
\numberwithin{proposition}{section}
\def\no{\|}
\def\11{{\rm 1~\hspace{-1.4ex}l} }

\newcommand{\la}{\lambda}

\newcommand{\beq}{\begin{equation}}
\newcommand{\eeq}{\end{equation}}
\newcommand{\ben}{\begin{eqnarray}}
\newcommand{\een}{\end{eqnarray}}
\newcommand{\beno}{\begin{eqnarray*}}
\newcommand{\eeno}{\end{eqnarray*}}
\def\R{\mathbb R}

\def\T{\mathbb T}

\def\ba{\begin{aligned}}
\def\ea{\end{aligned}}
\def\be{\begin{equation}}
\def\ee{\end{equation}}
\def\ben{\begin{align*}}
\def\enn{\end{align*}}
\def\R{\mathbb{R}}

\def\T{\mathcal{T}}

\def\ra{\rightarrow}

\def\a{\alpha}

\def\ld{\lambda}

\def\sg{\sigma}
\def\w{\omega}
\def\e{\varepsilon}
\def\d{\delta}

\def\no{\nonumber}
\def\q{\quad}
\def\lt{\left}
\def\rt{\right}

\begin{document}
\selectlanguage{english}
\title[3D Kakeya inequality]
{ On Wolff's $L^{\frac{5}{2}}-$Kakeya maximal inequality in $\R^3$
}

\author{Changxing Miao}
\address{Institute of Applied Physics and Computational Mathematics\\}
\email{miao\_{}changxing@iapcm.ac.cn}
\author{Jianwei Yang}
\address{The Graduate School of China Academy of Engineering Physics\\} \email{geewey.young@gmail.com}
\author{Jiqiang Zheng}
\address{The Graduate School of China Academy of Engineering Physics.}
\email{zhengjiqiang@gmail.com}

\begin{abstract}
We reprove Wolff's $L^{\frac{5}2}-$ bound for the $\R^3-$Kakeya maximal function
without appealing to the argument of induction on scales.
The main ingredient in our proof is an adaptation of  Sogge's strategy
used in the work on Nikodym-type sets in curved spaces.
Although the equivalence between these two type maximal functions
is well known, our proof may shed light on some new geometric observations
which is interesting in its own right.
\end{abstract}
\maketitle
\selectlanguage{english}
\tableofcontents


\section{Introduction }

Let $\d>0,\ \xi\in S^2,\ a\in \R^3$. Define a $\d-$tube centered at $a$ in direction of
$\xi$ as
$$
T^\d_\xi(a)=\Bigl\{x\in \R^3\,\Big|\;|(x-a)\cdot\xi|\leq \frac{1}2,\ |(x-a)^\bot|\leq \d\Bigr\},
$$
where $x^\bot=x-(x\cdot\xi)\xi$ and $S^2$ denotes the standard unit
two sphere in $\R^3$.

Let $f:\R^3\ra \mathbb{C}$ be a locally integrable function and
define the Kakeya maximal operator as
\begin{equation}\label{def-kekaya}
    f^*_\d(\xi)=\sup_{a\in \R^3}\frac{1}{|T^\d_{\xi}(a)|}\int_{T^\d_{\xi}(a)}|f(x)|dx.
\end{equation}
we naturally extend this definition homogeneously by letting
$$
f^*_\d(\eta)=f^*_\d\bigl(\frac{\eta}{|\eta|}\bigr), \forall~
\eta\neq 0.
$$
In particular, we have
for $\ld>0$,
$$f^*_\d(\ld\xi)\mathrel{\mathop=^{\rm def}}f^*_\d(\xi),\;\xi\in S^2.$$
A longstanding conjecture about the Kakeya maximal function is for $1\leq p\leq 3$
\begin{equation}\label{k-conj}
    \|f^*_\d\|_{L^p(S^2)}\lesssim_\e\d^{-\frac{3}{p}+1-\e}\|f\|_{L^p(\R^3)},\q\forall\,\e>0.
\end{equation}
This implies immediately the Kakeya sets  in $\R^3$ have full
Hausdorff dimension.

If $p=1$, \eqref{k-conj} becomes trivial since
$$
\| f^*_\d\|_{L^1(S^2)}\leq |S^2|\|f^*_\d\|_{L^\infty(S^2)}\lesssim\d^{-2}\|f\|_{L^1}.
$$
By interpolation, \eqref{k-conj} is equivalent to the end-point estimate
\begin{equation}\label{k-conj-end}
     \|f^*_\d\|_{L^3(S^2)}\lesssim_\e \d^{-\e}\|f\|_{L^3(\R^3)}.
\end{equation}
\begin{remarque}
In general, the conjecture about the estimates on Kakeya maximal function asserts that for all dimensions there holds
\begin{equation}\label{k-conj-end}
     \|f^*_\d\|_{L^d(S^{d-1})}\lesssim_\e \d^{-\e}\|f\|_{L^d(\R^d)}.
\end{equation}
Consequently, this implies the Hausdorff dimension of Kakeya sets in $\R^d$ should be exactly $d$.
For later use, we define $C_{\d,d}$ to be
\begin{equation}\label{C}C_{\d,d}=\sup_{\|f\|_{L^2}\neq0}\|f^*_\d\|_{L^{2}(S^{d-1})}/\|f\|_{L^2(\R^d)}.\end{equation}
\end{remarque}
In the case when $d=2$, \eqref{k-conj-end} is valid (see \cite{ref Bourgain1} and \cite{ref Cordoba}).
However for $d\geq3$, the question remains open and becomes extremely difficult. At the early stages,
some primitive results with $p=\frac{d+1}2$ can be deduced easily, see \cite{ref Bourgain3},
 \cite{ref Christ. D. R}, \cite{ref K-T2} and \cite{ref Wolff}.
The breakthrough in this direction was obtained by Bourgain \cite{ref Bourgain1}
through establishing an inductive formula for the $L^p-$ estimates on Kakeya maximal functions with $p=\frac{d+1}2+c_d$
and $0<c_d<\frac12$. This result was improved by Wolff \cite{ref Wolff} to $p=\frac{d+2}{2}$.
Several subsequent progresses on $d\geq4$ were made by Bourgain \cite{ref Bourgain2}, Katz and Tao \cite{ref K-T} and Tao-Vargas-Vega \cite{ref T-V-V}.
 We refer to the investigations in \cite{ref Bourgain3}, \cite{ref K-T2}, \cite{ref Tao 1} and \cite{ref wolff2}
for further references and historical remarks.

In this paper, we focus on the three dimensional case.
The best result in $\R^3$ is hitherto due to Wolff \cite{ref Wolff}.

\begin{theoreme}[T. Wolff, 1995]\label{thm 1}
The Kakeya maximal function \eqref{def-kekaya} satisfies the following estimate
\begin{equation}\label{wolff}
    \| f^*_\d\|_{L^{\frac{10}{3}}(S^2)}\lesssim_\e\d^{-\frac{1}{5}-\e}\|f\|_{L^{\frac{5}{2}}(\R^3)}.
\end{equation}
\end{theoreme}
\begin{remarque}
From this estimate, \eqref{k-conj} follows immediately with $p=\frac5{2}$.
\end{remarque}

As discussed above, Wolff's approach combines the induction on
scales and the ideas from combinatorics. It belongs, on the whole,
to the category of geometric method, which is fairly efficient in
dealing with low dimensional cases as pointed out in \cite{ref
K-T2}. This work is aimed at better understanding the geometric
combinatorial behavior of the Kakeya maximal function in $\R^3$, and
the purpose of this article is to prove \eqref{wolff} without using
induction on scales. The main idea is inspired by Sogge's strategy
on Nikodym-type sets in 3-dimensional manifolds with constant
curvatures \cite{ref Sogge}. By exploring this method and combining
the ideas from Bourgain-Guth's  multilinear approach to oscillatory
integrals \cite{ref Bourgain Guth}, we believe it is possible to
obtain some  improvements on the known results of the Kakeya
problems.

In order to prove \eqref{wolff}, it suffices to show the following restricted weak type maximal estimate
(see \cite{ref Wolff} or the appendix)
\begin{equation}\label{wolf-red}
    \|f^*_\d\|_{L^{\frac{10}3,\infty}(S^2)}\lesssim_\e \d^{-\frac{1}5-\e}\|f\|_{L^{\frac{5}2,1}},
\end{equation}
which is the core of this paper.


This paper is organized as follows.
In Section 2, we introduce some terminologies
of the scheme on account of the multiplicities of the tubes associated to the discrete version of \eqref{wolf-red}.
In Section 3, we obtain an $L^2-$type estimate for an auxiliary maximal function in $\R^d$ in terms of the $(d-1)-$dimensional Kakeya maximal functions.
 Section 4 is devoted to a crucial Lemma  4.3, which reduces our ultimate goal \eqref{g_dis_form} to a generic condition \eqref{4.7}.
Finally, we verify this condition  for $d=3$ in Section 5 and
complete the proof of Theorem \ref{thm 1}.
 For the sake of self-completeness , we show the local property of the conjecture \eqref{k-conj-end} as well as the implication of
 \eqref{wolf-red} to \eqref{wolff} in the appendix.

\section{Preliminaries on the multiplicity argument}
As was discussed before, we only need to prove \eqref{wolf-red}. Since the problem is local \footnote{See \cite{ref Bourgain1}
or the  Appendix at the end of this paper.}, a standard averaging argument in \cite{ref Bourgain1} yields the equivalent form  of \eqref{wolf-red}
\begin{equation}\label{distri_est}
    \sigma\{\xi\in S^2:(\chi_E)^*_\d(\xi)\geq\lambda\}\lesssim_\e\lt(\lambda^{-\frac{5}2}\d^{-(\frac{1}2+\e)}|E|\rt)^{\frac{4}{3}},\,
     \forall\, \ld\in[\d,1],
\end{equation}
where $E$ is a subset of the unit ball $B(0,1)$.

Let $A_\ld=\{\xi\in S^2:(\chi_E)^*_\d(\xi)\geq\lambda\}$. By dividing $S^2$  into the finite union of caps, where the total number of
these caps is independent of $\d$, we may assume that $A_\ld$ is contained in a cap with the aperture angle less than one. The discretization
of \eqref{distri_est} is achieved by choosing
a maximal $\d-$separated subset $\{\xi^\nu\}^M_{\nu=1}$ of $ A_\ld$ such that \eqref{distri_est} is equivalent to
\begin{equation}\label{discrete_form}
    M\d^2\lesssim_\e\lt(\ld^{-\frac{5}{2}}\d^{-\frac{1}2-\e}|E|\rt)^{\frac{4}3},\q\forall\, \ld \in [\d,1].
\end{equation}

By definition of $(\chi_E)^*_\d$, we have for each
$\nu\in\{1,\ldots, M\}$, there is a tube
$T^\d_{\xi^\nu}(a_\nu):=T^\d_\nu$ satisfying $ |E\cap
T^\d_\nu|\geq\frac{\ld}{2}|T^\d_\nu|. $ We shall use these tubes to
set up our multiplicity argument. Since  this argument works for all
dimensions, we set it up in the sequel for general $d\geq 3$, and
apply it to the case $d=3$ at the end of our proof.

Notice that the higher dimensional counterpart of \eqref{wolff}
reads (see \cite{ref Wolff})
\begin{equation}\label{g_wolff}
    \|f^*_{\d}\|_{L^{\frac{(d-1)(d+2)}{d}}(S^{d-1})}\lesssim_\e\d^{-\frac{2d}{d+2}+1-\e}\|f\|_{L^\frac{d+2}{2}(\R^d)},
\end{equation}
the analogue for \eqref{discrete_form}  becomes for  $d\geq 3$
\begin{equation}\label{g_dis_form}
    M\d^{d-1}\lesssim_\e\lt(\d^{-\e}\frac{|E|\ld^{-p}}{\d^{d-p}}\rt)^{\frac{q}{p}},
\end{equation}
with $p=\frac{d+2}{2}$ and $q=\frac{(d-1)p}{p-1}$.

Now we introduce some preliminaries for the modified multiplicity argument. Fix $x\in B(0,1)\subset\R^d$ and
$j\in\{1,\ldots,M\}$.  We define for $\theta,\sigma\in [\d,1]$
\begin{align}\label{I}
\mathcal{I}_{\theta,\sigma}(x,j)\mathrel{\mathop=^{\rm def}}
&\Big\{i:\,\chi_{T^\d_i}(x)=1,\,\angle(T^\d_i,T^\d_j)\in \Big[\frac{\theta}{2},\theta\Big),
\\
\no&\q\q\q\q\q\q
\Big|T^\d_i\cap\Big\{y\in E:\text{dist}(y,\gamma_j)\in\Big[\frac{\sigma}{2},\sigma\Big)\Big\}\Big|
\geq\Big(2^4\log_2\frac{1}{\d}\Big)^{-1}\ld|T^\d_i|\Big\}.
\end{align}
where $\gamma_j$ is the central axis of the tube $T^\d_{j}$ and
$\angle(T^\d_i,T^\d_j):=\angle(\xi^i,\xi^j)$.

\begin{figure}[ht]
\begin{center}
$$\ecriture{\includegraphics[width=7cm]{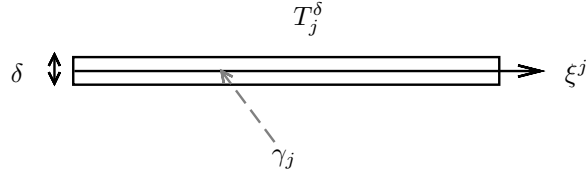}}
{\aat{-1}{7}{$\d$}\aat{23}{0}{$\gamma_j$}\aat{25}{12}{$T^\d_{j}$}\aat{50}{7}{$\xi^j$}}$$
\end{center}
\caption{$\gamma_j$ as the center of tube $T^\d_j$.}
\end{figure}

We consider the following two scenarios\footnote{See \cite{ref
wolff2} for the motivation from Szemeredi-Trotter's theorem.}.
\begin{itemize}
 \item I.(Low multiplicity scenario) Let $N_1$ be a nonnegative integer such that
 there are at least $\frac{M}{2}$ many $j$'s satisfying
$$
\Big|T^\d_j\cap E\cap\Big\{x\in \R^d:
\sum^M_{l=1,l\neq j}\chi_{T^\d_l}(x)\leq N_1\Big\}\Big|\geq
\frac{\ld}4|T^\d_j|;
$$
\end{itemize}
\begin{itemize}
 \item II$_{\theta\sigma}$.(High multiplicity at angle $\theta$ and
distance $\sigma$). Let $N_2$ be a nonnegative integer such that for
$\theta, \sigma\in[\d,1]$ and $\mathcal{I}_{\theta,\sigma}(x,j)$
defined as in \eqref{I}
\begin{align*}
 \text{Card}\Big\{j: \Bigl|T^\d_j\cap E\cap\Big\{x:\text{Card}\,\mathcal{I}_{\theta,\sigma}(x,j)\geq2^{-3}\Big(\log_2\frac{1}{\d}\Big)^{-2}N_2\Big\}
 \Bigr|
 \geq 2^{-3}\Big(2\log_2\frac{1}{\d}&\Big)^{-2}\ld|T^\d_j|\Big\}\\
 &\geq\frac{M}{2^4(\log_2\frac{1}{\d})^2}.
\end{align*}
\end{itemize}

It is easy to see that  $N_1\geq M$ is sufficient for scenario I.
If we denote by $N$ the smallest $N_1$ such that scenario I is valid, then there exist $\theta,\sigma\in[\d,1]$ such that
II$_{\theta\sigma}$ also holds  for $N_2=N$. Essentially, this is  achieved by using a dyadic pigeonhole principle. To see this,
by the minimality of $N$ and triangle inequality, we have at least $\frac M2+1$ many $j$'s such that
\begin{equation}\label{I'}
\Big|\mathcal{Q}^\d_{j}:=T^\d_j\cap E\cap \Bigl\{x:\sum^M_{l=1,l\neq j}\chi_{T^\d_l}\geq N\Bigr\}\Big|\geq\frac{\ld}{4}|T^\d_j|.
\end{equation}
For any $x\in\mathcal{Q}^\d_j$, we have
\begin{equation}
\label{lowerbound}\sum^M_{l=1,l\neq j}\chi_{T^\d_l}(x)\geq N,
\end{equation}
and
\begin{equation}\label{claim1}
    \Bigl\{k: k\neq j,\,x\in T^\d_k\Bigr\}\subset\bigcup^{[\log_2\frac{1}{\d}]+1}_{\nu=1}\Bigl\{i:x\in T^\d_i,\angle(T^\d_j,T^\d_i)\in [2^{\nu-1}\d,
    2^{\nu}\d)\Bigr\}.
\end{equation}
On the other hand, we claim that
\begin{align}\label{claim2}
    &\{k: k\neq j,\,x\in T^\d_k\}\\
    \nonumber&\subset\bigcup^{[\log_2\frac{2}\d]}_{\nu'=1}\Big\{i:x\in T^\d_i,\Big|T^\d_i\cap\{y\in E:\text{dist}(y,\gamma_j)\in [2^{\nu'-1}\d,
    2^{\nu'}\d)\}\Big|\geq \bigl(2^4\log_2\frac{1}\d\bigr)^{-1}\ld|T^\d_i|\Big\}.
\end{align}
On account of \eqref{claim1} and \eqref{claim2}, we may write
\begin{align}
    \nonumber&\{k: k\neq j,\,x\in T^\d_k\}\\
    \nonumber&\subset\bigcup^{[\log_2\frac{1}\d]+1}_{\nu=1}\bigcup^{[\log_2\frac{2}\d]}_{\nu'=1}\Big(\{i:x\in T^\d_i,\angle(T^\d_j,T^\d_i)\in
    [2^{\nu-1}\d,2^{\nu}\d)\}\\
    \nonumber&\q\q\q\q\q\q\q\cap\Big\{i:x\in T^\d_i,\Big|T^\d_i\cap\{y\in E:\text{dist}(y,\gamma_j)\in [2^{\nu'-1}\d,2^{\nu'}\d)\}\Big|\geq
    \bigl(2^4\log_2\frac{1}\d\bigr)^{-1}\ld|T^\d_i|\Big\}\Big)\\
    \nonumber&\subset\bigcup^{[\log_2\frac{1}\d]+1}_{\nu=1}\bigcup^{[\log_2\frac{2}\d]}_{\nu'=1}\mathcal{I}_{2^\nu\d,2^{\nu'}\d}(x,j).
\end{align}
In view of \eqref{lowerbound}, we have at least $N$ many tubes $T^\d_k$ containing $x$ such that $k\neq j$. By choosing $\d\ll0.01$, we have
$$N\leq 2^3 \Big(\log_2\frac{1}{\d}\Big)^2\sup_{\substack{1\leq \nu\leq[ \log_2\frac{1}\d]+1\\1\leq \nu'\leq [\log_2{\frac{2}{\d}}]}}\text{Card}\;
\mathcal{I}_{2^\nu\d,2^{\nu'}\d}(x,j).$$
Therefore, there are $\nu$ and $\nu'$, which may depend on $x$ and $j$, such that
$$\text{Card}\;\mathcal{I}_{2^\nu\d,2^{\nu'}\d}(x,j)\geq 2^{-3}\Big(\log_2\frac{1}{\d}\Big)^{-2}N.$$
From the above discussions, we have
$$
\mathcal{Q}^\d_{j}\subset\bigcup^{[\log_2\frac{1}\d]+1}_{\nu=1}\bigcup^{[\log_2\frac{2}{\d}]}_{\nu'=1}\Big(T^\d_j\cap E\cap \Bigl\{x:\text{Card}\;
\mathcal{I}_{2^\nu\d,2^{\nu'}\d}(x,j)\geq 2^{-3}\Big(\log_2\frac{1}{\d}\Big)^{-2}N\Bigr\}\Big),
$$
which, by \eqref{I'}, yields
$$
\frac{\ld}4|T^\d_j|\leq 2^3\Big(\log_2\frac{1}\d\Big)^2\sup_{\nu,\nu'}\Big|T^\d_j\cap E\cap \Bigl\{x:\text{Card}\;\mathcal{I}_{2^\nu\d,2^{\nu'}\d}(x,j)
\geq 2^{-3}\Big(\log_2\frac{1}{\d}\Big)^{-2}N\Bigr\}\Big|.
$$
Consequently, we have found $\nu=\nu(j)$ and $\nu'=\nu'(j)$ such that
\begin{equation}\label{pigeon}
\Big|T^\d_j\cap E\cap \Big\{x:\text{Card}\;\mathcal{I}_{2^\nu\d,2^{\nu'}\d}(x,j)\geq2^{-3}\Big(\log_2\frac{1}{\d}\Big)^{-2}N\Big\}\Big|\geq 2^{-3}
\ld\Bigl(2\log_2\frac1\d\Bigr)^{-2}|T^\d_j|
\end{equation}

Since there are at most $2^4\Big(\log_2\frac{1}\d\Big)^2$ many pairs of $(\nu,\nu')$'s and at least $\frac M2+1$ many $j$'s as in \eqref{pigeon},
by pigeonhole's principle there is a pair $(\nu_0,\nu'_0)$ such that II$_{\theta\sigma}$ holds for $\theta=2^{\nu_0}\d$ and $\sigma=2^{\nu'_0}\d$.

It remains  to prove \eqref{claim2}. For $k\neq j$, we have
\begin{align*}
\frac{\ld}2|T^\d_k|
\leq|T^\d_k\cap E|\leq 2&\sum^{[\log_2\frac{2}{\d}]}_{\nu'=1}\Bigl|T^\d_k\cap E\cap\Bigl\{y:\text{dist}(y,\gamma_j)\in[2^{\nu'-1}\d,2^{\nu'}\d)\Bigr\}
\Bigr|\\
\leq 8& \log_2{\frac{1}{\d}}\q\sup_{\nu'}\Big|T^\d_k\cap E\cap\Bigl\{y:\text{dist}(y,\gamma_j)\in[2^{\nu'-1}\d,2^{\nu'}\d)\Bigr\}\Big|
\end{align*}
where we have used the fact that $k\neq j$ implies $\angle (T^\d_k,T^{\d}_j)\geq c \d$ for some $c>0$ suitably large.
Thus \eqref{claim2} follows.

\begin{remarque}The high and low multiplicity scenarios for tubes was first exploited by
Wolff \cite{ref Wolff}. This along with the the argument of induction on scales improves significantly the bound on Kakeya type maximal functions.
The modified version in the above form was in spirit of  Sogge \cite{ref Sogge}. Combining this with an $L^2-$estimate for an auxiliary maximal function,
 one may establish the Nikodym type maximal inequality in curved background with constant curvatures.
\end{remarque}

\section{An auxiliary maximal function inequality}

Let $\gamma_j$ be the central axis of $T^\d_j$ as shown in Figure 1.
We may assume without loss of generality that $\gamma_j$ is parallel
to $ e_1$, where $\{e_1,e_2,\ldots,e_d\}$ is the orthogonal normal
basis of $\R^d$. For $y\in\R^d$, denote by $y=(y_1,y')$ with
$y'=(y_2,\ldots,y_d)$. In this section, we always assume that $f$ is
an integrable function $\R^d$ supported in the hollow cylinder
$\{y\in\R^d: |y_1|\leq 1,\, \frac{\sg}{2}\leq|y'|\leq\sg\}$.

For any $\xi \in A_\ld$ and a tube $T_{\xi}^{\delta}$ in the
direction of $\xi$ such that  $\angle(\xi,\xi^j)>0$ and
$T^\d_{j}\cap T^\d_{\xi}\neq\emptyset$, there is a unique point
$q=q(j,\,\xi)$ such that
\begin{equation}\label{q}\text{dist}(q,\gamma_j)+\text{dist}(q,\gamma_\xi)
=\min_{x\in\R^d}\Bigl[\text{dist}(x,\gamma_j)+\text{dist}(x,\gamma_\xi)\Bigr],\end{equation}
 where $\gamma_\xi$
is the central axis of the tube $T^\d_\xi$ in the direction $\xi$.
We denote by $\gamma_j\wedge\gamma_\xi$ the point $q$ such that
\eqref{q} holds. Let
$\w^{j}_\xi(y)=\big[\text{dist}(y,\gamma_j\wedge\gamma_\xi)\big]^\frac{1}{2}$.
For brevity, we  write $\w^j_{\xi^\nu}$ and $\gamma_{\xi^\nu}$
respectively as $\w^j_{\nu}$ and $\gamma_{\nu}$.

Define the auxiliary maximal function as
$$
A^\theta_{\d,j}(f)(\xi)=\sup_{\substack{T^\d_j\cap T^\d_\xi\neq\emptyset
\\\angle(T^\d_j,T^\d_\xi)\in[\frac{\theta}2,\theta]}}
\frac{1}{|T^\d_\xi|}\int_{T^\d_\xi}|f(y)|\w^j_\xi(y)dy,
$$
We define $A^\theta_{\d,j}(f)(\xi)$ to be zero if
$\angle(T^\d_j,T^\d_\xi)$ is outside the interval
$[\frac{\theta}2,\theta]$.

The difference between this auxiliary maximal function and $f^*_\d$ is that the supremum is taken under more  constraints for the tubes
in direction of $\xi$. Besides, we put a weight function for technical reasons. On one hand, it is clear that
$A^\theta_{\d,j}(f)(\xi)\lesssim f^*_\d(\xi)$ when $f$ is supported in a unit ball. On the other hand, a more interesting fact is that
we can estimate the $L^2$ norm of $A^\theta_{\d,j}(f)$ by means of $(d-1)-$dimensional Kakeya maximal functions. Thus,
 we reduce the problem of dimension $d$ to the problem of dimension $(d-1)$. In this sense, our argument is very similar to Bourgain's induction
 on dimension argument in \cite{ref Bourgain1}.
To be more specific, we prove in this section
\begin{proposition}\label{induct}
Let $A^\theta_{\d,j}(f)(\xi)$ be as above, we have for all $j$
\begin{equation}\label{L2-reduce}
\|A^\theta_{\d,j}(f)\|_{L^2(S^{d-1})}\leq 2^{10}C_{\d,d-1}\d^{-\frac{d-3}{2}}\|f\|_{L^2(\R^d)},
\end{equation}
where $C_{\d,d-1}$ is  as  in \eqref{C}.
\end{proposition}

\begin{proof}
Without loss of generality, we let $j=0, \xi^0=e_1$ and suppress the subscript $j$ and superscript $\theta$ in $A^\theta_{\d,j}$.
By symmetry, we only consider the following integral
\begin{equation}\label{3.1}
\int_{S^{d-1}_+}|A_\d(f)|^2(\xi)d\Sigma(\xi),
\end{equation}
where $d\Sigma$ represents the standard surface measure on the unit sphere and $$S^{d-1}_+=\{\xi\in S^{d-1}\mid \xi_1\geq 0\}.$$

\begin{figure}[ht]
\begin{center}
$$\ecriture{\includegraphics[width=6cm]{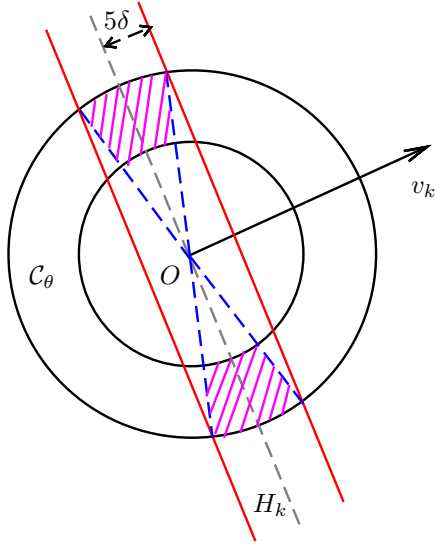}}
{\aat{13}{58}{$5\d$}\aat{46}{40}{$v_k$}\aat{19}{30}{$O$}\aat{5}{30}{$\mathcal{C}_\theta$}\aat{29}{5}{$H_k$}}$$
\end{center}
\caption{The angular decomposition for $\mathcal C_\theta$.}
\end{figure}

Since $\angle(\xi,e_1)\in[\frac{\theta}2,\theta]$,
we may restrict $\sin\frac{\theta}2\leq|\xi'|\leq\sin\theta$ in the integration of \eqref{3.1}  with respect to $\xi=(\xi_1,\xi')$.
Let
$$\mathcal{C}_\theta=\Bigl\{\xi'=(\xi_2,\ldots,\xi_d)\in\R^{d-1}:\sin\frac{\theta}2\leq|\xi'|\leq\sin\theta\Bigr\},$$
and take
a maximal $\frac{\d}{\theta}-$separated subset $\{v_k\}^{\sim(\theta/\d)^{d-2}}_{k=1}$ of $S^{d-2}$, which is the unit
sphere in $\R^{d-1}_{\xi'}$. Define
$$
\Pi^{\d,\theta}_k=\Big\{ \xi'\in\mathcal{C}_\theta:\Bigl|\Bigl\langle \frac{\xi'}{|\xi'|},v_k\Bigr\rangle\Bigr|\leq\frac{\d}2\Big\},
$$
which is  contained in a $5\d-$neighborhood of the
$(d-2)-$dimensional hyperplane $H_k$ perpendicular to $v_k$. Next,
we define $\Gamma^{\d,\theta}_1=\Pi^{\d,\theta}_1,$ and
$\Gamma^{\d,\theta}_k=\Pi^{\d,\theta}_k\setminus\Big(\bigcup^{k-1}_{j=1}
\Pi^{\d,\theta}_j\big)$ for $k\geq 2$. Then we have
$\mathcal{C}_\theta\subset\bigcup_k\Gamma^{\d,\theta}_k$ and
$\Gamma^{\d,\theta}_k\cap\Gamma^{\d,\theta}_{k'}=\emptyset$ for
$k\neq k'$.

If $\xi'\in\Gamma^{\d,\theta}_k$ for some $k\in\Bigl\{1,\ldots,\sim \Big(\frac{\theta}{\d}\Big)^{d-2}\Bigr\}$, then the tube $T^\d_\xi$, in
 direction of $\xi=(\sqrt{1-|\xi'|^2},\xi')\in S^{d-1}$ must lie in a $50\d-$neighborhood $\tilde{H}^{50\d}_k$
 of the hyperplane $\tilde{H}_k:=\text{span}\{e_1,H_k\}$, since $T^\d_{\xi^0}\cap T^\d_\xi\neq \emptyset$.

\begin{figure}[ht]
\begin{center}
$$\ecriture{\includegraphics[width=5cm]{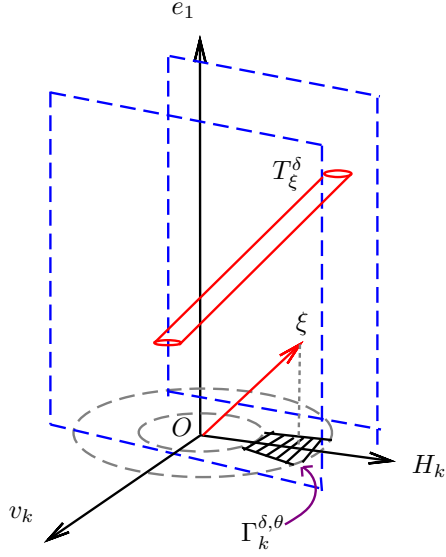}}
{\aat{19}{72}{$e_1$}\aat{32}{50}{$T^\d_\xi$}\aat{35}{30}{$\xi$}\aat{28}{2}{$\Gamma^{\d,\theta}_k$}
\aat{19}{16}{$O$}\aat{-2}{5}{$v_k$}\aat{50}{11}{$H_k$}}$$
\end{center}
\caption{$T^\d_\xi$ is contained in $\tilde{H}^{50\d}_k$.}
\end{figure}
From this observation, we introduce the following cylindrical sets
$$
\mathcal{V}_k=\{y\in\R^d: |y_1|\leq 1,\, |\langle y',v_k\rangle|<50\d\}.
$$
Then we have the following almost orthogonality estimate
\begin{equation}\label{y_ortho}
\sum_k\chi_{\mathcal{V}_k\cap\text{supp}f}(y)\leq C\frac{\theta^{d-2}}{\d^{d-3}\sg}.
\end{equation}
To see this, for any $y'$ such that $\frac{\sg}2\leq|y'|\leq\sg$ and
denote by $H^{50\d}_k$ the $50\d-$neighborhood of $H_k$. Let
$\Pi_{y'}$ be the hyperplane in $\R^{d-1}$ perpendicular to $y'$.
One easily verifies that $H^{50\d}_k$ contains $y'$ only when
$v_k\in S^{d-2}$ lives in a $\frac{100\d}{\sg}-$neighborhood of
$\Pi_{y'}$. Thus there are at most
$O\bigl(\frac{\theta^{d-2}}{\sg\d^{d-3}}\bigr)$ many $H^{50\d}_k$'s
containing $y'$ simultaneously.

Now we turn to estimate \eqref{3.1}. This will be reduced to the following maximal function $\mathcal{A}_\d$ defined similar to $A_\d$,
$$
\mathcal{A}_\d(f)(\xi)\mathrel{\mathop=^{\rm def}}
\sup_{\substack{T^\d_0\cap T^\d_\xi\neq\emptyset\\\angle(T^\d_0,T^\d_\xi)\in[\frac{\theta}2,\theta]}}
\frac{1}{|T^\d_\xi|}\int_{T^\d_\xi}|f(y)|dy.
$$
For the moment, we assume that for each $k\in \Big\{1,\ldots,\sim\Big(\frac{\theta}{\d}\Big)^{d-2}\Big\}$
\begin{equation}\label{3.2}
\|\mathcal{A}_\d(f\chi_{\mathcal{V}_k})\|_{L^2(\{\xi\in
S^{d-1}_+\mid\,\xi'\in\Gamma^{\d,\theta}_k\})}\leq C_{\d,d-1}
\|f\chi_{\mathcal{V}_k}\|_{L^2}.
\end{equation}
We next deduce \eqref{L2-reduce} under the assumption \eqref{3.2}.
Noting that for $\theta\leq 1$,
$$\frac{1}{\sqrt{1-\sin^2\theta}}\leq 2,$$
and
$$\w_\xi(y)\sim \Big(\frac{\sg}{\theta}\Big)^{\frac{1}2},\q\forall y\in\mathcal{V}_k\cap T^\d_\xi\cap\text{supp} f,\;\forall\xi'\in
\Gamma^{\d,\theta}_k,$$
we estimate \eqref{3.1} in the following manner
\begin{align*}
\eqref{3.1}&\leq4\int_{\mathcal{C}_\theta}|A_\d(f)|^2(\sqrt{1-|\xi'|^2},\xi')d\xi'
\leq4\sum_k\int_{\Gamma^{\d,\theta}_k}|A_\d(f\chi_{\mathcal{V}_k})|^2(\xi')d\xi'\\
&\lesssim \frac{\sg}{\theta}\sum_k\int_{\Gamma^{\d,\theta}_k}|\mathcal{A}_\d(f\chi_{\mathcal{V}_k})|^2(\xi')d\xi'
\lesssim \frac{\sg}{\theta}C^2_{\d,d-1}\sum_k\int_{\R^{d}}|f|^2\chi_{\mathcal{V}_k}(y)dy
\lesssim C^2_{\d,d-1}\Big(\frac{\theta}{\d}\Big)^{d-3}\|f\|^2_2,
\end{align*}
where the last inequality is due to \eqref{y_ortho}.
\begin{figure}[ht]
\begin{center}
$$\ecriture{\includegraphics[width=5cm]{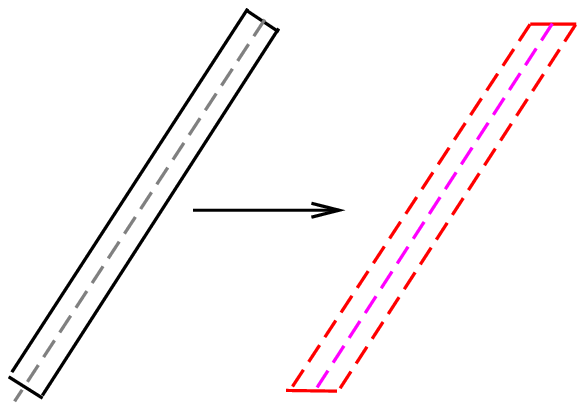}}
{\aat{-7}{5}{$T^\d_\xi$}\aat{36}{5}{$\mathcal{T}^\d_\xi$}}$$
\end{center}
\caption{$T^\d_\xi$;\, and\,  $\mathcal{T}^\d_\xi$.}
\end{figure}

Therefore, we are reduced to proving \eqref{3.2}.
By rotation invariance, we may assume $k=1$ and $v_1$ is identical to $e_d$.
We may assume further that $f$ is  supported in $\mathcal{V}_1$.
Clearly, $\Gamma^{\theta,\d}_1$ is contained in the region
$$\Theta^{\theta,\d}_1:=\Big\{\xi'\in\R^{d-1}:|(\xi_2,\ldots,\xi_{d-1})|\leq\sin\theta,\,|\xi_d|\leq 10\d\Big\}.$$
Fix $\xi'\in\Theta^{\theta,\d}_1$ and denote by $p\in\gamma_\xi$ such that $p$ is closest to $\gamma_\xi\wedge\gamma_0$ with $ p=(p_1,p')$.
We slightly modify $T^\d_\xi(a)$ to be $\mathcal{T}^\d_\xi(a)$ as follows, singling out $y_1$ as the parameter of the central axis (see Figure 4)
\begin{align*}
\T^\d_\xi=\Big\{(y_1,y')\in \R\times
\R^{d-1}:&\,\Big|y'-p'-\frac{y_1-p_1}{\sqrt{1-|\xi'|^2}}\xi'\Big|\leq \frac{\d}{2\sqrt{1-|\xi'|^2}}
,\\
&p_1-(\frac{1}{2}-\text{dist}(a,p))\cos\a\leq y_1\leq p_1+(\frac{1}{2}+\text{dist}(a,p))\cos\a\Big\}.
\end{align*}
where $a=(a_1,a')$ is the middle of $\gamma_\xi$ and $\a:=\angle(\gamma_0,\gamma_\xi)$.

Let $\mathcal{P}(y_d)$ be the hyperplane perpendicular to $v_1$ and parameterized by $y_d$. Fix $y_d\in[-50\d,50\d]$ and consider
$\mathcal{P}(y_d)\bigcap\T^\d_\xi:=\mathcal{E}_{\d}(y_d)$. One can verify that $\mathcal{E}_\d$ is an ellipse with major axis at least $1/10$.
In fact, let $\beta$ be the angle between  $\T^\d_\xi$ and $v_1$.
\begin{figure}[ht]
\begin{center}
$$\ecriture{\includegraphics[width=5cm]{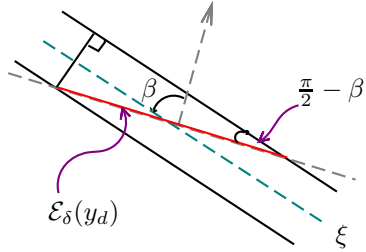}}
{\aat{20}{21}{$\beta$}\aat{40}{21}{$\frac{\pi}{2}-\beta$}\aat{8}{5}{$\mathcal{E}_\d(y_d)$}\aat{45}{2}{$\xi$}}$$
\end{center}
\caption{The ellipsoid $\mathcal{E}_\d(y_d)$.}
\end{figure}
We have $\cos\beta=\xi_d$, which implies the major axis is at least
$\frac{\d}{\sin(\frac{\pi}{2}-\beta)}\geq\frac{\d}{|\xi_d|}\geq\frac{1}{10}$.
Thus $\mathcal{E}_\d(y_d)$ can be regarded as a $(d-1)-$dimensional
Kakeya tube with dimensions
$1\times\underbrace{\d\times\ldots\times\d}_{d-2}$.

 Let $r=(1-\xi^2_d)^{\frac{1}{2}}$ and $\xi^{''}=(\xi_2,\ldots,\xi_{d-1})$. Since $|\xi^{''}|\leq \sin \theta\leq\frac{\sqrt{3}}{2}$ and $|r-1|\ll 1$
 by taking $\d$ sufficiently small, we see that $(\sqrt{r^2-|\xi^{''}|^2},\xi^{''})$ represents a vector
on $r S^{d-2}$. By Fubini's theorem, the integral average of $f$ over $\T^\d_\xi$ is controlled by
$$
\d^{-(d-1)}\int_{|y_d|\leq 50\d}dy_d\int_{\mathcal{E}_\d(y_d)}|f(y_1,\ldots,y_{d-1},y_d)|dy_1\ldots dy_{d-1}.
$$
Next, we use the $(d-1)-$dimensional Kakeya maximal functions to  bound the above formula. In particular, this implies
$$
\mathcal A_\d(f)(\xi')\lesssim\d^{-1}\int_{|y_d|\leq 50\d}M_\d(f(\ldots,y_d))(\sqrt{r^2-|\xi^{''}|^2},\xi^{''})dy_d,
$$
where $M_\d(f(\ldots,y_d))$ denotes the $(d-1)-$dimensional Kakeya
maximal operator acting on $f$, and  $f$ is regarded as a function
of the $d-1$ variables $(y_1,\ldots,y_{d-1})$ with $y_d$ frozen as a
parameter.

Using Minkowski's inequality and H\"{o}lder's inequalities, we obtain by  $r<1$
\begin{align*}
&\lt(\int_{|(\xi_2,\ldots,\xi_{d-1})|\leq\sin\theta}|\mathcal{A}_\d(f)(\xi')|^2d\xi_2\ldots d\xi_{d-1}\rt)^{\frac{1}2}\\
&\leq\d^{-1}\int_{|y_d|\leq 50\d}\lt(\int|M_\d(f(\ldots,y_d))|^2(\sqrt{r^2-|\xi^{''}|^2},\xi^{''})d\xi_2\ldots d\xi_{d-1}\rt)^{\frac{1}2}dy_d\\
&\leq2\d^{-1}\int_{|y_d|\leq 50\d}\lt(\|M_\d(f(\ldots,y_d))\|^2_{L^2(S^{d-2})}\rt)^{\frac{1}2}dy_d\\
&\leq2\d^{-1}C_{\d,d-1}\int_{|y_d|\leq 50\d}\|f(\ldots,y_d)\|_{L^2_{y_1,\ldots,y_{d-1}}}dy_d\\
&\leq2^6\d^{-1/2}C_{\d,d-1}\|f\|_2.
\end{align*}
Squaring both sides and integrating with respect to $\xi_d\in[-10\d,10\d]$, we get \eqref{3.2} and hence \eqref{L2-reduce}.
\end{proof}

It is well-knownthat $C_{\d,2}=\log\frac{1}{\d}$ \footnote{See formula (1.5) in \cite{ref Bourgain1} for example.} , and consequently we conclude
\begin{corollaire}\label{coral_L2}
If $d=3$, we have for some $c>0$
\begin{equation}\label{3d_l2}
\|A^\theta_{\d,j}(f)\|_{L^2(S^2)}\leq c \Big(\log\frac{1}{\d}\Big)\|f\|_{L^2(\R^3)}.
\end{equation}
\end{corollaire}
This corollary is crucial in the proof of Theorem \ref{thm 1}.
\begin{remarque}
We observe some essential distinctions between the 3D and higher
dimensional problems. Indeed, we find in Proposition \ref{induct}
that the loss of the factor $\d^{-\frac{d-3}{2}}$ vanishes in the
three dimensional case. This allows us to use the optimal estimates
on 2D Kakeya maximal function to deduce Wolff's $L^{\frac52}-$bound
on the 3D case. On the other hand, we do not know whether the
$\d^{-\frac{d-3}{2}}$ loss is necessary in \eqref{L2-reduce}. Since
our method of reducing the estimate on $d$-dimensional auxiliary
maximal function to the estimates of $(d-1)-$dimensional Kakeya
maximal function is rather crude, it seems possible by strengthening
the argument to reduce the $\frac{d-3}{2}-$exponent of the loss.
This might be easier when $d$ is large, while for lower dimensions,
it seems rather difficult.
\end{remarque}

\section{The key Lemmas}
\begin{lemme}\label{lemma 4.1_i}
Let $N$ satisfy scenario I, then
$|E|\geq\ld M\d^{d-1}(16N)^{-1}$.
\end{lemme}

\begin{proof}
Relabeling the subscripts, we may write the tubes involved in case I as $\{T^\d_j\}^K_{j=1}$ with $M\geq K\geq M/2$.
Then, we have
\begin{align*}
\frac{\ld M\d^{d-1}}{8N}
&\leq\frac{\ld}{4N}\sum^K_{j=1}|T^\d_j|
\leq\frac{1}N\sum^K_{j=1}\Bigl|T^\d_j\cap E\cap\Bigl\{x\in\R^d\mid\sum^M_{\ell=1,\ell\neq j}\chi_{T^\d_\ell}(x)\leq N\Bigr\}\Bigr|\\
&\leq\int_{ E\cap\bigl\{x\in\R^d:\,\sum\limits_{\ell=1,\ldots,M }\chi_{T^\d_\ell}(x)\leq N+1\bigr\}}
\frac{1}{N}\sum^K_{j=1}\chi_{T^\d_j}(x)dx
\leq 2|E|.
\end{align*}
\end{proof}

\begin{lemme}\label{lemma 4.2}
Suppose there are $M$ many tubes $\{T^\sg_j\}^M_{j=1}$ such that
$j\neq j'$ and $T^\sg_j\cap T^\sg_{j'}\neq \emptyset$ implies $\angle (T^\sigma_j,T^\sg_{j'})\geq \gamma$ for some $0<\gamma<\frac\pi2$.
Assume also that for some $\rho>0$ and any $a\in \R^d$, there are $M_0$ many of such tubes satisfying
\begin{equation}\label{4.1}
\rho|T^\sg_j|\leq \Big|T^\sg_j\cap E\cap B(a,\sg/\gamma)^c\Big|.
\end{equation}
Then we have
\begin{equation}\label{lemma 4.2 eq}
|E|\geq\rho \sg^{d-1}M^{1/2}_0/2.
\end{equation}
\end{lemme}
\begin{proof}
By relabeling the indices, we have, under these assumptions, a sequence $\{T^\sg_j\}^{M_0}_{j=1}$ satisfying
$$
\rho \sg^{d-1}M_0\leq \int_{E}\sum^{M_0}_{j=1}\chi_{T^\sg_j}(x)dx.
$$
Thus, there exists an $x_0\in E$ such that
$$
\sum^{M_0}_{j=1}\chi_{T^\sg_j}(x_0)\geq\frac{\rho\sg^{d-1}M_0}{2|E|}.
$$
We relabel the subcollection of the tubes
$\{T^\sg_j\}_{j\in\{1,\ldots,C_*\}}$ containing $x_0$, where
$$C_*=C_{\rho,\sg,M_0,E}=\Big[\frac{\rho\sg^{d-1}M_0}{2|E|}\Big].$$

We notice the orthogonality outside the ball $B(x_0,\sg/\gamma)$  by
the following observation. It follows
 from the angle condition in the assumptions that the component of
$T^\sg_j\cap T^\sg_{j'}$ must be contained in the ball $B(x_0,L)$
with $L$ at most $\frac{\sg}{2}/\sin{\frac{\gamma}{2}}$, which is
less than $\sg/\gamma$ for $\gamma<\frac \pi2$. With the help of
this orthogonality,  the choice of $C_*$ and \eqref{4.1}, we have
\begin{align*}
|E|&\geq\Bigl|E\cap B(x_0,\sg/\gamma)^c\cap\bigcup^{C_*}_{j=1}T^\sg_j\Bigr|\geq\sum^{C_*}_{j=1}|E\cap B(x_0,\sg/\gamma)^c\cap T^\sg_j|\\
&\geq C_* \rho\sg^{d-1}\geq\frac{\rho^2\sg^{2(d-1)}M_0}{4|E|},
\end{align*}
where we use Lemma \ref{lemma 4.1_i} in the last inequality.

\end{proof}

\begin{figure}
\begin{center}
$$\ecriture{\includegraphics[width=7cm]{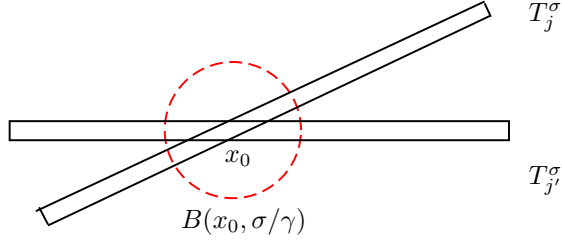}}
{\aat{50}{20}{$T^\sg_j$}\aat{22}{7}{$x_0$}\aat{18}{1}{$B(x_0,\sg/\gamma)$}\aat{50}{5}{$T^\sg_{j'}$}}$$
\end{center}
\caption{The orthogonality of tubes outside a ball $B(x_0,\sg/\gamma)$.}
\end{figure}

\begin{lemme}\label{lemma 4.3}
Let $N$ satisfy both case I and case II$_{\theta\sg}$. Then, there
are $M2^{-4}\Big(\log_2\frac{1}\d\Big)^{-2}$ many tubes $T^\d_j$  in
II$_{\theta\sg}$. Suppose for any $\e>0$, there exists $C_\e>0$ such
that for small $\d>0$ and any point $a\in\R^d$,
\begin{equation}\label{4.7}
|E\cap B(a,\d^\e\ld^{d-2})^c\cap T^\sg_j|\geq C_\e\ld^3\sg\d^{d-2+\e}N,
\end{equation}
then we have \eqref{g_dis_form}.
\end{lemme}

\begin{proof}
We rewrite \eqref{g_dis_form} as $|E|^2\geq C_\e\ld^{d+2}\Big(\d^{d-1}M\Big)^{\frac{d}{d-1}}\d^{d-2+\e}$. Then it suffices to prove
\begin{equation}\label{4.5}
|E|\geq \ld M\d^{d-1}(16N)^{-1},
\end{equation}
and
\begin{equation}\label{4.6}
|E|\geq C_\e\ld^{d+1}N(\d^{d-1}M)^{\frac{1}{d-1}}\d^{d-2+\e},
\end{equation}
where \eqref{4.5} is proved in Lemma \ref{lemma 4.1_i} and it remains to prove \eqref{4.6}.
\vskip 0.2cm
Let $\{\xi_j\}_{j\in\{1,\ldots,[M2^{-4}(\log_2\frac{1}{\d})^{-2}]\}}$ be the directions of $T^\d_j$.
Noting that $\sg\geq \d$, we have $\gamma:=\frac{100 \sg}{\d^\e\ld^{d-2}}\geq\d^{1-\e}$ since $\ld\leq 1$ and $d\geq 3$.
If $\gamma\geq \frac\pi2$, then \eqref{4.6} follows immediately from \eqref{4.7}.\\

Otherwise, we can take a maximal $\gamma-$separated subsequence of $\{\xi_j\}$ and denote them by $\{\xi_{j_k}\}^{M_0}_{k=1}$.
By maximality, we obtain for some $C_2>0$
$$
M_0\geq C_2\frac{M}{2^4\Big(\log_2\frac{1}\d\Big)^2}\d^{d-1}\Big(\frac{\d^\e\ld^{d-2}}{\sg 100}\Big)^{d-1}\geq C_2\frac{M\d^{d-1 }}{2^4\Big(\log_2
\frac{1}\d\Big)^2}\Big(\frac{\d^\e\ld^{d-2}}{\sg 100}\Big)^{d-1}
.$$
 and
use Lemma \ref{lemma 4.2} with $\rho=C_\e\ld^3\sg^{2-d}\d^{d-2+\e}N$ as well as \eqref{4.7} to get
\begin{align*}
|E|&\geq C_\e \ld^3\sg^{2-d}\d^{d-2+\e}N\times\sg^{d-1}\times\frac{M^{\frac{1}{2}}_0}2\\
&\geq\frac{C_\e}2\ld^3\sg\d^{d-2+\e}N\times\Big(\frac{C_2}4 M\d^{d-1-\e c_3}\Big)^{\frac{1}{d-1}}\times\frac{\d^\e\ld^{d-2}}{\sg 100}\\
&\geq\tilde{C}_\e\ld^{d+1}\d^{d-2+(2+\frac{1}{d-1})\e}(M\d^{d-1})^{\frac{1}{d-1}}N,
\end{align*}
which implies \eqref{4.6}, since $\e>0$ is arbitrarily small.
\end{proof}

\begin{remarque}
In the second step, we have used $M^{\frac{1}2}_0\geq M^{\frac1{d-1}}_0$ for $d\geq 3$.
Since we can only verify \eqref{4.7} for $d=3$, this loss caused by cutting $\frac{1}{2}$ down to $\frac1{d-1}$ is dismissed.
However, this loss appears to be significant when one deals with the higher dimensional cases with $d\geq 4$.
\end{remarque}

\section{Completion of the proof to Theorem \ref{thm 1}}
In this section, we confine ourselves in the case when $d=3$ and prove \eqref{4.7} using Corollary \ref{coral_L2}. This will complete the proof of
Wolff's $L^{\frac5{2}}-$bound for Kakeya maximal functions.
Before proving \eqref{4.7}, we first prove a simplified version.
\begin{lemme}\label{lemma 4.3}
Let $d=3$ and $N$ satisfy both scenario I and  II$_{\theta\sg}$. Denote the $M2^{-4}\Big(\log_2\frac1{\d}\Big)^{-2}$ many tubes by $\{T^\d_j\}$
 in II$_{\theta\sg}$. For any $\e>0$, there exists a $C_\e>0$ such that for $\d>0$ sufficiently small, we have
\begin{equation}\label{simple4.7}
|E\cap T^\sg_j|\geq C_\e\ld^3\sg\d^{1+\e}N.
\end{equation}
\end{lemme}

\begin{proof}
For any $j\in\Big\{1,\ldots,\Big[M2^{-4}\Big(\log_2\frac1{\d}\Big)^{-2}\Big]\Big\}$, we define
$$
S^\d_j\mathrel{\mathop=^{\rm def}}T^\d_j\cap E\cap\lt\{x:\text{Card}\;\mathcal{I}_{\theta,\sg}(x,j)\geq 2^{-3}N\Big(\log_2\frac{1}\d\Big)^{-2}\rt\}.
$$
By definition of $\mathcal{I}_{\theta,\sg}(x,j)$, we see that there exists an $M_0\in(0,M]$ and a subcollection $\{T^\d_{i_k}\}^{M_0}_{k=1}$ of
$\{T^\d_i\}^M_{i=1}$ such that
\begin{equation}\label{(1)}
\angle(T^\d_{i_k},T^\d_j)\in\Big[\frac{\theta}2,\theta\Big),
\end{equation}

\begin{equation}\label{(2)}
\lt|T^\d_{i_k}\cap E\cap\lt\{y:\text{dist}(y,\gamma_j)\in\lt[\frac{\sg}{2},\sg\rt)\rt\}\rt|\geq \lt(2^4\log_2\frac{1}{\d}\rt)^{-1}\ld|T^\d_{i_k}|,
\end{equation}
and\footnote{It is a little tricky here. We first fix $j$ and $x\in S^\d_j$ then we get the subcollection with condition \eqref{(1)} and \eqref{(2)}.
However, this subcolletion may depend on $x$. In order to avoid this dependency, we consider all the possible subcollections, take their union and
denote $M_0$ as the total number of the tubes included,
then we are safe with our argument without causing confusions. }
\begin{equation}\label{(3)}
\Big(\sum^{M_0}_{k=1}\chi_{T^\d_{i_k}}\Big)\Big|_{S^\d_j}\geq \frac {N}{2^3}\lt(\log_2\frac1{\d}\rt)^{-2}.
\end{equation}
Moreover, we have from the definition of II$_{\theta,\sigma}$, \eqref{(3)} and $S^\d_j\subset T^\d_j$
\begin{align*}
2^{-3}\frac{\ld}{(4\log_2\frac{1}\d)^2}|T^\d_j|
&\leq|S^\d_j|\leq 2^3\bigl(2\log_2\frac{1}{\d}\bigr)^2N^{-1}\int_{T^\d_j}\sum^{M_0}_{k=1}\chi_{T^\d_{i_k}}(x)dx\\
&\leq N^{-1}2^3\Big(2\log_2\frac{1}\d\Big)^2\sum^{M_0}_{k=1}|T^\d_{i_k}\cap T^\d_j|\leq N^{-1}2^3\Big(\log_2\frac{1}\d\Big)^2 8\d^3M_0/\theta,
\end{align*}
where we have used
$|T^\d_{i_k}\cap T^\d_{j}|\leq \frac{\d^3}{\theta}$.
Hnece we conclude
\begin{equation}\label{equ5.4}
M_0\geq 2^{-10}\theta\d^{-1}N\ld\Big(\log_2\frac{1}\d\Big)^{-4}.
\end{equation}

\begin{figure}
\begin{center}
$$\ecriture{\includegraphics[width=6cm]{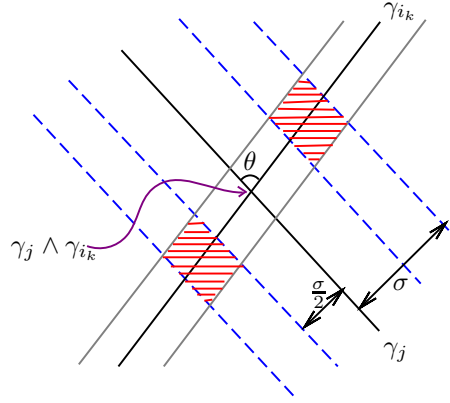}}
{\aat{41}{13}{$\sg$}\aat{32}{12}{$\frac{\sg}2$}\aat{25}{26}{$\theta$}\aat{0}{17}{$\gamma_j\wedge\gamma_{i_k}$}\aat{40}{6}{$\gamma_j$}\aat{40}{43}
{$\gamma_{i_k}$}}$$
\end{center}
\caption{ $T^\d_{i_k}\cap\{y:\text{dist}(y,\gamma_j)\in[\sg/2,\sg)\}$ is indicated by the shaded region.}
\end{figure}

Now, for any $T^\d_{i_k}$, we have ( see Figure 7 )
\begin{align}\label{net}
|T^\d_{i_k}|^{-1}&\int_{T^\d_{i_k}}\chi_{E\cap T^\sg_j}(y)\w^j_{i_k}(y)dy\\
\no&\geq |T^\d_{i_k}|^{-1}\int_{T^\d_{i_k}\cap E\cap\{y:\text{dist}(y,\gamma_j)\in[\sg/2,\sg)\}}[\text{dist}(y,\gamma_{i_k}\wedge\gamma_j)]^{\frac{1}2}dy
\\
\no&\geq\lt(\frac{\sg}\theta\rt)^{\frac{1}2}|T^\d_{i_k}|^{-1}\cdot|T^\d_{i_k}\cap E\cap\{y:\text{dist}(y,\gamma_j)\in[\sg/2,\sg)\}|\\
\no&\geq\Big(2^4\log_2\frac{1}\d\Big)^{-1}\lt(\frac{\sg}\theta\rt)^{\frac{1}2}\ld.
\end{align}
On the other hand,
$$|T^\d_{i_k}|^{-1}\int_{T^\d_{i_k}}\chi_{E\cap T^\sg_j}(y)\w^j_{i_k}(y)dy\leq A^\theta_{\d,j}(\chi_{E\cap T^\sg_j})(\xi_{i_k}).$$
Squaring both sides, multiplying $\d^2$ and summing up with respect to $k=1,\ldots,M_0$, we have
\begin{align*}
M_0\d^2\Big(2\log_2\frac{1}\d\Big)^{-2}\frac{\ld^2\sg}\theta&\leq\sum^{M_0}_{k=1}\Big|A^\theta_{\d,j}(\chi_{E\cap T^\sg_j})(\xi_{i_k})\Big|^2\d^2\\
&\lesssim\int_{S^2}\Big|A^\theta_{\d,j}(\chi_{E\cap T^\sg_j})(\xi)\Big|^2d\Sigma(\xi)\\
&\lesssim\bigl(\log\frac{1}{\d}\bigr)|E\cap T^\sg_j|,
\end{align*}
where the last step involves the $L^2-$estimate  \eqref{3d_l2}.

Invoking the lower bound \eqref{equ5.4}, we obtain
\eqref{simple4.7}.
\end{proof}

\begin{proposition}\label{final lemma}
If $d=3$, then \eqref{4.7} holds.
\end{proposition}

\begin{proof}
For $i\in \mathcal{I}_{\theta,\sg}(x,j)$, we have by choosing $\d$ small
\begin{align*}
\Big|T^\d_i\cap\Big\{y\in E\cap B(a,\d^\e\ld)^c&:\text{dist}(y,\gamma_j)\in[\sg/2,\sg]\Big\}\Big|\\
&\geq\Big(2^4\log_2\frac{1}\d\Big)^{-1}\ld|T^\d_i|-\d^\e\ld|T^\d_i|\geq\Big(2^5\log_2\frac{1}\d\Big)^{-1}\ld|T^\d_i|.
\end{align*}
If we define
\begin{align*}
\mathcal{\tilde{I}}_{\theta,\sigma}(x,j)\mathrel{\mathop=^{\rm def}}
&\Big\{i:\,\chi_{T^\d_i}(x)=1,\,\angle(T^\d_i,T^\d_j)\in \Big[\frac{\theta}{2},\theta\Big),
\\
\no&\q\q\q\q\q\q
\Big|T^\d_i\cap\Big\{y\in E\cap B(a,\d^\e\ld)^c:\text{dist}(y,\gamma_j)\in\Big[\frac{\sigma}{2},\sigma\Big)\Big\}\Big|
\geq\Big(2^5\log_2\frac{1}{\d}\Big)^{-1}\ld|T^\d_i|\Big\},
\end{align*}
then, clearly $\mathcal{I}_{\theta,\sg}(x,j)\subset\mathcal{\tilde{I}}_{\theta,\sg}(x,j)$, which gives
$\text{Card}\;\mathcal{I}_{\theta\sg}(x,j)\leq\text{Card}\;\mathcal{\tilde{I}}_{\theta\sg}(x,j)$.
Since there are at least $M2^{-4}\Big(\log_2\frac{1}\d\Big)^{-2}$ many $j$'s satisfying II$_{\theta\sg}$, we have for each such $j$
$$
2^{-3}\Big(4\log_2\frac{1}\d\Big)^{-2}\ld|T^\d_j|\leq\lt|\lt\{x\in T^\d_j\cap E\cap B(a,\d^\e\ld)^c:\text{Card}\,\mathcal{\tilde{I}}_{\theta,\sg}(x,j)
\geq2^{-3}\Big(\log_2\frac{1}\d\Big)^{-2}N\rt\}\rt|+\d^\e\ld|T^\d_j|.
$$
Taking $\d$ small, we obtain for this $j$
$$
2^{-3}\Big(4\log_2\frac{1}\d\Big)^{-2}\frac\ld{2}|T^\d_j|\leq\lt|\lt\{x\in T^\d_j\cap E\cap B(a,\d^\e\ld)^c:\text{Card}\,\mathcal{\tilde{I}}_{\theta,\sg}
(x,j)
\geq\Big(2\log_2\frac{1}\d\Big)^{-2}N\rt\}\rt|.
$$
Replacing $E$ in lemma \ref{lemma 4.3} with $E\cap B(a,\d^\e\ld)^c$ and using \eqref{simple4.7} with $\ld/2$  instead of $\ld$, we finally conclude
\eqref{4.7} for $d=3$. Therefore, we complete the proof of of our main theorem.
\end{proof}

\section{Appendix}

\subsection{The local property of Kakeya maximal function inequality}
In this section, we shall see the problem on Kakeya maximal inequality is local.
Namely, to derive \eqref{wolf-red}, we can assume $f$ is supported in a ball of finite size.
In particular, we may assume $f$ is supported in the unit ball centered at zero.
To show that the general inequality \eqref{k-conj-end} for $f$ defined on $\R^d$ follows from its localized version, we first choose
 a maximal $\d-$separated subset $\{\xi^k\}_{k\in\mathfrak K}$ in $S^{d-1}$ with ${\rm Card} \mathfrak K\sim \d^{-(d-1)}$,
  and write for a locally integrable function $f$
\begin{align}\label{6.1}
\int_{S^{d-1}}|f^*_\d(\xi)|^q d\Sigma(\xi)
\lesssim\sum_{k\in\mathfrak K}\int_{\angle(\xi,\xi^k)\leq \d}|f^*_\d(\xi)|^q d\Sigma(\xi).
\end{align}

\begin{figure}[ht]
\begin{center}
$$\ecriture{\includegraphics[width=2cm]{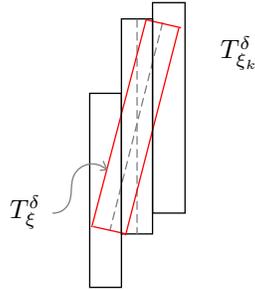}}
{\aat{-10}{26}{$T^\d_\xi$}\aat{58}{79}{$T^\d_{\xi_k}$}}$$
\end{center}
\caption{  $T^\d_{\xi}$ is covered by the translates of $T^\d_{\xi_k}$.}
\end{figure}
Since $\angle (\xi,\xi_k)\leq\d$, there is a $c=c(d)>0$ independent of $\d$ such that $T^\d_\xi$ is covered by
a union of at most $c$ many parallel translates of the tube $T^\d_{\xi^k}$ ( see Figure 8). Moreover, there are two uniform constants
$c_1,c_2$ depending only on $d$ such that
$$
c_1 f^*_\d(\xi^k)\leq f^*_\d(\xi)\leq c_2 f^*_\d(\xi^k),\q \forall \angle (\xi,\xi^k)\leq \d.
$$
Hence \eqref{6.1} is bounded up to some constant depending only on $d$ by
\begin{equation}\label{discretized}
\sum_{k\in\mathfrak K}|f^*_\d(\xi^k)|^q\d^{d-1}.
\end{equation}
By definition of $f^*_\d(\xi^k)$, there is a tube $T^\d_k:=T^\d_{\xi^k}(a_k)$ in direction of $\xi^k$
such that
$$\frac1{T^\d_k}\int_{T^\d_k}|f|(y)dy\geq\, \frac 12\, f^*_\d(\xi^k).$$
Similarly for any $\xi\in S^{d-1}$ with $\angle(\xi,\xi_k)\leq \d$, there is a tube $T^\d_\xi(a)$ so that
$$
\frac{1}{| T^\d_\xi|}\int_{T^\d_\xi}|f|(y)dy\geq\,\frac 12\,f^*_\d(\xi).
$$
Now, we take a maximal $1-$seperated subset of $\{a_k\}_{k\in\mathfrak K}$. After relabeling the indices, we may denote this subsequence
by $\{a_j\}^J_{j=1}$ with $J\leq {\rm Card } \mathfrak K$. Thus, for any $k\in \mathfrak K$, there is some $j\in\{1,\ldots,J\}$ such that
 $|a_k-a_j|\leq 1$, and hence  $T^\d_k\subset B(a_j,2)$. Based on this observation, we may write
\begin{align}
\nonumber\eqref{discretized}\lesssim&\sum^J_{j=1}\sum_{k:|a_k-a_j|\leq 1}|\bigl(f\chi_{B(a_j,2)}\bigr)^*_\d(\xi^k)|^q\d^{d-1}\\
\nonumber\lesssim&\sum^J_{j=1}\sum_{k:|a_k-a_j|\leq 1}\int_{\angle(\xi,\xi^k)\leq\d}|\bigl(f\chi_{B(a_j,2)}\bigr)^*_\d(\xi)|^q d\Sigma(\xi)\\
\label{7.3}\lesssim & \sum^J_{j=1}\int_{S^{d-1}}|\bigl(f\chi_{B(a_j,2)}\bigr)^*_\d(\xi)|^q d\Sigma(\xi).
\end{align}

For $q\geq p$, assume that $\|f^*_\d\|_{L^q(S^{d-1})}\lesssim_\e \d^{-\frac dp+1-\e}\|f\|_{L^p(B(a,2))}$ for all $a\in\R^d$.
We have by finite overlaps of the balls $\{B(a_j,2)\}^J_{j=1}$ and Minkowski's inequality
\begin{align*}
\eqref{7.3}\lesssim_\e & \d^{-(\frac dp-1)q-q\e}  \sum^J_{j=1}\Bigl(\int_{\R^d}|\bigl(f\chi_{B(a_j,2)}\bigr)(x)|^p dx\Bigr)^{\frac qp}\\
\lesssim_\e&\d^{-(\frac dp-1)q-q\e} \|f\|^q_{L^p(\R^d)}.
\end{align*}
This yields the same estimate for general $f$.

\subsection{The implication of \eqref{wolf-red} to \eqref{wolff}}

As pointed in \cite{ref wolff2}, Drury \cite{Dr} had shown the following estimate
\begin{equation}
\|f_\delta^\ast\|_{L^{d+1}(S^{d-1})}\leq C_\varepsilon\delta^{-\frac{d-1}{d+1}-\varepsilon}\|f\|_{L^\frac{d+1}2(\R^d)}.
\end{equation}
We will use this fact as well as the following two estimates
\begin{equation}
\begin{cases}
\|f_\delta^\ast\|_{L^\infty(S^{d-1})}\leq \|f\|_{L^\infty(\R^d)},\\
\|f_\delta^\ast\|_{L^{p,\infty}(S^{d-1})}\leq C_\varepsilon\delta^{-\frac{d}q+1-\varepsilon}\|f\|_{L^{q,1}(\R^d)},~p=(d-1)q',
\end{cases}
\end{equation}
to derive
\begin{equation}
\|f_\delta^\ast\|_{L^{p}(S^{d-1})}\leq C_\varepsilon\delta^{-\frac{d}q+1-\varepsilon}\|f\|_{L^{q}(\R^d)},~p=(d-1)q'.
\end{equation}

We summarize this as the following lemma.
\begin{lemme}
Assume $T$ is a sublinear operator, $1\ll A, B<\infty$ and for
$p=(d-1)q',~q>\frac{d+1}2,$
\begin{align}\label{js1}
\|Tf\|_{L^\infty(S^{d-1})}\leq& \|f\|_{L^\infty(\R^d)},\\\label{js2}
\|Tf\|_{L^{d+1}(S^{d-1})}\leq& A\|f\|_{L^\frac{d+1}2(\R^d)},\\\label{js3}
\|Tf\|_{L^{p,\infty}(S^{d-1})}\leq& B\|f\|_{L^{q,1}(\R^d)},
\end{align}
then for any $\e>0$, there holds that
\begin{equation}\label{jl}
\|Tf\|_{L^{p}(S^{d-1})}\leq BA^\e\|f\|_{L^{q}(\R^d)}.
\end{equation}
\end{lemme}

\begin{proof}
We write $f=f_1+f_2+f_3$ with
$$f_1=f\chi_{|f|<\frac{\lambda}3},~f_2=f\chi_{|f|>A^\alpha\lambda},~f_3=f\chi_{\frac{\lambda}3\le|f|\le A^\alpha\lambda},~\alpha=\frac{2q}{d+1}-1.$$

From the layer cake
representation theorem  in \cite{Lieb}, we obtain
\begin{align*}
\|Tf\|_{L^p(S^{d-1})}^p
=&p\int_0^{+\infty}\lambda^{p-1}\nu\big(\{|Tf|>\lambda\}\big)d\lambda\\
\leq&p\int_0^{+\infty}\lambda^{p-1}\Big[\nu\big(\{|Tf_1|>\lambda/3\}\big)+\nu\big(\{|Tf_2|>\lambda/3\}\big)+\nu\big(\{|Tf_3|>\lambda/3\}
\big)\Big]d\lambda\\
\triangleq& I_1+I_2+I_3.
\end{align*}
It is easy to see that $I_1=0$ since $\nu\big(\{|Tf_1|>\lambda/3\}\big)=0$ by \eqref{js1}.
To estimate $I_2$, we use  \eqref{js2} to deduce that
\begin{equation}
\nu\big(\{|Tf_2|>\lambda/3\}\big)\lesssim\frac{A^{d+1}}{\lambda^{d+1}}\|f\|_{L^\frac{d+1}2}^{d+1}.
\end{equation}
This together with the trivial estimate
$$\nu\big(\{\xi\in S^{d-1}:~|Tf_2(\xi)|>\lambda/3\}\big)\lesssim1$$
implies that
\begin{equation}
\nu\big(\{|Tf_2|>\lambda/3\}\big)\lesssim\frac{A^{k}}{\lambda^{k}}\|f\|_{L^\frac{d+1}2}^{k},~0\le k\le d+1.
\end{equation}
Hence, we get by Minkowski's inequality
\begin{align*}
I_2=&p\int_0^{+\infty}\lambda^{p-1}\nu\big(\{|Tf_2|>\lambda/3\}\big)d\lambda
\lesssim A^k\int_0^{+\infty}\lambda^{p-1-k}\|f_2\|_{L^\frac{d+1}2}^{k}d\lambda\\
\lesssim& A^k\bigg(\int_{\R^d}|f|^\frac{d+1}2\Big(\int_0^{+\infty}\lambda^{p-1-k}\chi_{|f|>A^\alpha\lambda}d\lambda\Big)^\frac{d+1}{2k}
\bigg)^{\frac{2k}{d+1}}\\
\lesssim& A^kA^{-\alpha(p-k)}\Big(\int_{\R^d}|f|^{\frac{d+1}2p}dx\Big)^{\frac{2k}{d+1}}
\simeq\|f\|_{L^q}^p,
\end{align*}
where we have used $k=\frac{d+1}{2q}p$ and $\alpha=\frac{2q}{d+1}-1$ in the last step.

Finally, we turn to estimate $I_3$. By \eqref{js3} and the characterization of $L^{p,q}$ spaces, one has
\begin{equation}
\nu\big(\{|Tf_3|>\lambda/3\}\big)\leq\frac{B^p}{\lambda^p}\|f_3\|_{L^{q,1}}^p\lesssim \frac{B^p}{\lambda^p}(1+\alpha\log A)^{p-\frac{p}q}
\|f_3\|_{L^q}^p.
\end{equation}
Therefore, we estimate by Minkowski's inequality
\begin{align*}
I_3=&p\int_0^{+\infty}\lambda^{p-1}\nu\big(\{|Tf_3|>\lambda/3\}\big)d\lambda
\lesssim B^p(1+\alpha\log A)^{p-\frac{p}q}\int_0^{+\infty}\lambda^{-1}\|f_3\|_{L^q}^{p}d\lambda\\
\lesssim& B^p(1+\alpha\log A)^{p-\frac{p}q}\bigg(\int_{\R^d}|f|^q\Big(\int_0^{+\infty}\lambda^{-1}\chi_{\frac{\lambda}3\le|f|\le A^\alpha
\lambda}d\la\Big)^\frac{q}{p}\bigg)^{\frac{p}{q}}\\
\lesssim& B^p(1+\alpha\log A)^{p-\frac{p}q}(\log A)\|f\|_{L^q}^p\\
\lesssim& B^p A^\varepsilon\|f\|_{L^q}^p.
\end{align*}
Collecting all these estimates on $I_1,I_2$ and $I_3$, we obtain
$$\|Tf\|_{L^p(S^{d-1})}^p\le I_1+I_2+I_3\lesssim (1+B^p A^\varepsilon)\|f\|_{L^q}^p.$$
This  concludes the lemma.

\end{proof}

\subsection*{Acknowledgments}
 The authors thank the referee and the
associated editor for their invaluable comments and suggestions
which helped improve the paper greatly.  This work is supported in part by the NSF of China under grant No.11171033, No.11231006, and No.11371059.   C. Miao is also supported  by Beijing Center for Mathematics and Information
Interdisciplinary Sciences.


\end{document}